\documentclass[12pt]{amsart}
\usepackage{amsfonts}
\usepackage{amsfonts,latexsym,rawfonts,amsmath,amssymb,amsthm}
\usepackage{amsmath,amssymb,amsthm,amscd,esint}
\usepackage[plainpages=false]{hyperref}
\usepackage{appendix}

\usepackage{graphicx}

\RequirePackage{color}

\numberwithin{equation}{section}

\newcommand{\beq}{\begin{equation}}
\newcommand{\eeq}{\end{equation}}
\newcommand{\beqs}{\begin{eqnarray*}}
\newcommand{\eeqs}{\end{eqnarray*}}
\newcommand{\beqn}{\begin{eqnarray}}
\newcommand{\eeqn}{\end{eqnarray}}
\newcommand{\beqa}{\begin{array}}
\newcommand{\eeqa}{\end{array}}

\newcommand{\D}{\nabla}
\newcommand{\de}{\delta}

\newcommand{\p}{\partial}

\newcommand{\Om}{\Omega}
\newcommand{\e}{\epsilon}

\newcommand{\con}{\rightarrow}

\newtheorem{prop}{Proposition}[section]
\newtheorem{theo}[prop]{Theorem}
\newtheorem{lem}[prop]{Lemma}

\newtheorem{cor}[prop]{Corollary}

\allowdisplaybreaks
\title{A singular Moser-Trudinger inequality for mean value zero functions in dimension two}
\author{Xiaobao Zhu}

\address{Xiaobao Zhu:
School of Mathematics, Renmin University of China, Beijing 100872, China.}
\email{zhuxiaobao@ruc.edu.cn}

\thanks {*This research is partially supported by NSFC 11721101 and 11401575.}



\begin{document}
 \maketitle

\begin{abstract}
Let $\Omega\subset\mathbb{R}^2$ be a smooth bounded domain with $0\in\partial\Omega$. In this paper, we prove that for any $\beta\in(0,1)$, the supremum
$$\sup_{u\in W^{1,2}(\Omega), \int_\Omega u dx=0, \int_\Om|\nabla u|^2dx\leq1}\int_\Omega \frac{e^{2\pi(1-\beta) u^2}}{|x|^{2\beta}}dx$$
is finite and can be attained. This partially generalizes a well-known work of Alice Chang and Paul Yang \cite{CY88} who have obtained the inequality when $\beta=0$.
\end{abstract}

\baselineskip=16.4pt
\parskip=3pt
\section{Introduction}

Let $\Omega\subset\mathbb{R}^2$ be a smooth bounded domain. The famous Moser-Trudinger inequality, which was first proposed by Trudinger \cite{Tr}
and then sharpened by Moser \cite{M}, states that
\begin{align}\label{ineq-mt}
\sup_{u\in W^{1,2}_0(\Omega),\int_\Om|\nabla u|^2dx\leq1} \int_{\Omega} e^{4\pi u^2} dx < +\infty,
\end{align}
where $4\pi$ is the best constant.
When $\Omega$ is a disc, Carleson and Chang \cite{CC} proved that the supremum in (\ref{ineq-mt}) can be attained. Struwe \cite{St} showed that the
result remains true when $\Omega$ is close to a disc. Then Flucher \cite{F} generalized the result to arbitrary domain in $\mathbb{R}^2$.

To study conformal deformation of metrics on $\mathbb{S}^2$, Chang and Yang \cite{CY88} generalized (\ref{ineq-mt}) to mean value zero functions. Precisely, they obtained
\begin{align}\label{ineq-cy}
\sup_{u\in W^{1,2}(\Omega),\int_{\Omega}udx=0, \int_\Om|\nabla u|^2dx\leq1} \int_{\Omega} e^{2\pi u^2} dx < +\infty,
\end{align}
where $2\pi$ is the best constant. It was then proved by Yang \cite{Y07} that the supremum in (\ref{ineq-cy}) can be attained.

Using a rearrangement argument, Adimurthi and Sandeep \cite{AS} derived a singular version of (\ref{ineq-mt}). Suppose $0\in\Omega$, for any $\beta\in(0,1)$,
there holds
\begin{align}\label{ineq-as}
\sup_{u\in W^{1,2}_0(\Omega), \int_\Om|\nabla u|^2dx\leq1} \int_{\Omega} \frac{e^{4\pi(1-\beta) u^2}}{|x|^{2\beta}} dx < +\infty,
\end{align}
where $4\pi(1-\beta)$ is the best constant. It was proved by Csat\'{o} and Roy \cite{CR} that the supremum in (\ref{ineq-as}) can be attained.
Based on blow-up analysis, Yang and the author \cite{YZ} gave a different proof.\\

In this paper, motivated by Adimurthi and Sandeep's inequality (\ref{ineq-as}), we want to give a singular version of Chang and Yang's inequality (\ref{ineq-cy}).
Precisely, we shall prove
\begin{theo} \label{smt}
Let $\Omega\subset\mathbb{R}^2$ be a smooth bounded domain with $0\in\p\Om$.
For any $\beta\in(0,1)$, we have
\begin{align}\label{ineqsmt}
\sup_{u\in W^{1,2}(\Omega),\int_\Omega u dx=0,\int_{\Omega}|\nabla u|^2dx\leq1}\int_{\Omega} \frac{e^{2\pi(1-\beta) u^2}}{|x|^{2\beta}} dx < +\infty
\end{align}
and the supremum can be attained, where $2\pi(1-\beta)$ is the best constant.
\end{theo}

%

Our proof is based on blow-up analysis. The method we shall use was first originated in the celebrated article of Ding, Jost, Li and Wang \cite{DJLW},
then developed by Li \cite{Li}. Now it becomes very useful and standard in the study of Moser-Trudinger inequalities,
we refer the readers to \cite{Li2,LL,LLY,Y06,Y07,Y2007,Y15,Z} and references therein.

To end the introduction, we would like to outline the history of
the study of Moser-Trudinger inequalities. It is known to all that $W^{1,2}_0(\Omega)\hookrightarrow L^p(\Omega)$ for any $p\geq1$, but
$W_0^{1,2}(\Omega)\not\hookrightarrow L^{\infty}(\Omega)$. The classical Moser-Trudinger inequality (\ref{ineq-mt}) fills in this gap. Therefore, as a limiting case
of Sobolev embedding, inequality (\ref{ineq-mt}) plays an important role in analysis. Besides, inequality (\ref{ineq-mt}) is also non-substitutable in geometry especially
related with the famous prescribed Gaussian curvature problem. Now, we talk about it.
Let $(\Sigma,g)$ be a smooth and compact Riemannian surface without boundary, Fontana \cite{Fo} generalized inequality (\ref{ineq-mt}), he proved
\begin{align}\label{ineq-f}
\sup_{u\in W^{1,2}(\Sigma,g),\int_\Sigma udv_g=0,\int_\Sigma |\nabla_g u|^2dv_g\leq1}\int_\Sigma e^{4\pi u^2}dv_g<\infty.
\end{align}
Here, $4\pi$ in inequality (\ref{ineq-f}) is sharp and is the same constant as in inequality (\ref{ineq-mt}). However, when the metric $g$ has conical singularities, with divisor $\textbf{D}=\sum_{j=1}^m\beta_jp_j$ where $\beta_j>-1$ and $p_j\in\Sigma$ are different with each other, the constant in corresponding Moer-Trudinger inequality is relevant with the divisor. Precisely, let $(\Sigma,g)$ be a compact Riemannian surface without boundary and $g$ representing the divisor $\textbf{D}$, denote by $b_0=4\pi\min\{\min_{1\leq j\leq m}(1+\beta_j),1\}$, Troyanov \cite{Tro} and Chen \cite{Ch} proved that
\begin{align}\label{ineq-tc}
\sup_{u\in W^{1,2}(\Sigma,g),\int_\Sigma udv_g=0,\int_\Sigma |\nabla_g u|^2dv_g\leq1}\int_\Sigma e^{b_0u^2}dv_g<\infty,
\end{align}
where $b_0$ is sharp. This phenomenon was then observed by Adimurthi and Sandeep \cite{AS} for bounded domains which contain the origin. With the help of inequalities
(\ref{ineq-f}) and (\ref{ineq-tc}), one can understand the prescribed Gaussian curvature problem on surfaces very well, for this topic, we refer the reader to \cite{B,M,M2,KW,CY87,CY88,CD,ChL,Tro,CL3,St,YZ2,Z} and references therein.

The rest of this paper is organized as follows: In section 2, we prove inequality (\ref{ineqsmt}).
In section 3, we prove the supremum in (\ref{ineqsmt}) can be attained. Throughout this paper, we do not distinguish sequence and its subsequence.

\vskip 20pt

\section{Moser-Trudinger inequality}

In this section, we prove inequality (\ref{ineqsmt}). Firstly, by constructing a Moser function we show that
$2\pi(1-\beta)$ is the best possible constant. Secondly, following a variational argument we show that the supremum of subcritical singular Moser-Trudinger functional
$\int_{\Omega}|x|^{-2\beta}e^{2\pi(1-\beta-\e)u^2}dx$ can be attained by some $u_\e$. Thirdly, using the method of blow-up analysis we derive an upper bound
for the supremum of the critical singular Moser-Trudinger functional $\int_{\Omega}|x|^{-2\beta}e^{2\pi(1-\beta)u^2}dx$.


\subsection{$2\pi(1-\beta)$ is the best possible constant}
Without loss of generality, we assume $B_{2\delta}^+(0)\subset\Omega$ and $\partial\Omega$ is flat near $0$. For $0<l<\delta$, we define a Moser function
\begin{align*}
u_l(x):=\frac{1}{\sqrt{\pi}} \begin{cases}
\sqrt{\log\frac{\delta}{l}} \ \ \ \ &x\in B^+_l(0),\\
\frac{\log\frac{\delta}{|x|}}{\sqrt{\log\frac{\delta}{l}}} \ \ \ \ &x\in B_{\delta}^+(0)\setminus B^+_l(0),\\
\varphi C_l \ \ \ \ &x\in \Omega\setminus B_{\delta}^+(0),
\end{cases}
\end{align*}
where $\varphi$ is a cut-off function, which satisfies $0\leq\varphi\leq1$, $\varphi\equiv1$ in $\Omega\setminus B_{2\delta}^+(0)$,
$\varphi\equiv0$ in $B_\delta^+(0)$, and $|\nabla\varphi|\leq C/\delta$, $C_l$ is a constant to be determined later.
To ensure $\int_\Omega u_ldx=0$, we set
$$C_l=\frac{-\frac{\pi}{4}(\delta^2-l^2)}{\sqrt{\log\frac{\delta}{l}}\int_{\Omega\setminus B^+_\delta(0)}\varphi dx}=o_l(1) (l\rightarrow0+).$$
Calculating directly, one has
\begin{align*}
\int_\Omega |\nabla u_l|^2dx=1+\frac{C_l^2}{\pi}\int_{B^+_{2\delta}(0)\setminus B^+_\delta(0)}|\nabla \varphi|^2dx=1+o_l(1) (l\rightarrow0+).
\end{align*}
Then for any $\alpha>2\pi(1-\beta)$,
\begin{align*}
\int_\Omega |x|^{-2\beta}e^{\alpha\big(\frac{u_l}{||\nabla u_l||_{L^2(\Omega)}}\big)^2}dx
\geq& \int_{B_l^+(0)} |x|^{-2\beta}e^{\alpha\frac{\frac{1}{\pi}\log\frac{\delta}{l}}{1+o_l(1)}}dx\\
\geq& \frac{\pi}{2(1-\beta)}\delta^{\frac{\alpha}{\pi(1+o_l(1))}}l^{2(1-\beta)-\frac{\alpha}{\pi(1+o_l(1))}}\\
\rightarrow&+\infty \ \ \text{as}\ \ l\rightarrow0+.
\end{align*}
For any $0<\alpha<2\pi(1-\beta)$, we have
\begin{align*}
\frac{2\beta}{1-\frac{\alpha}{2\pi}}<2,
\end{align*}
then combining with (\ref{ineq-cy}), one obtains
\begin{align}\label{func-sub}
    &\sup_{u\in W^{1,2}(\Omega),\int_\Omega udx=0,\int_\Omega |\nabla u|^2dx\leq1}\int_{\Omega}|x|^{-2\beta}e^{\alpha u^2}dx\\
\leq& \sup_{u\in W^{1,2}(\Omega),\int_\Omega udx=0,\int_\Omega |\nabla u|^2dx\leq1}\left(\int_{\Omega}e^{2\pi u^2}dx\right)^{\frac{\alpha}{2\pi}}
       \left(\int_\Omega |x|^{-\frac{2\beta}{1-\frac{\alpha}{2\pi}}}dx\right)^{1-\frac{\alpha}{2\pi}}\nonumber\\
<& +\infty.\nonumber
\end{align}
Therefore, $2\pi(1-\beta)$ is the best possible constant in (\ref{ineqsmt}).

\subsection{Subcritical singular Moser-Trudinger functionals can be attained}
%
%

\begin{lem}\label{sub-mt} For any $\epsilon\in(0,1-\beta)$, there exists some
$u_\epsilon\in W^{1,2}(\Omega)$ with $\int_{\Omega}u_\epsilon dx=0$ and $\int_{\Om}|\D u_\e|^2 dx=1$ such that
\begin{align}\label{eq-sub}
   \int_{\Omega} \frac{e^{2\pi(1-\beta-\epsilon) u_{\epsilon}^2}}{|x|^{2\beta}} dx
=\sup_{u\in W^{1,2}(\Omega),\int_\Om udx=0, \int_\Om |\D u|^2dx\leq1} \int_{\Omega} \frac{e^{2\pi(1-\beta-\epsilon) u^2}}{|x|^{2\beta}} dx.
\end{align}
\end{lem}

\begin{proof}
For any $\e\in(0,1-\beta)$, we take $\{u_{\e,j}\}_{j=1}^{\infty}$ with $\int_\Om u_{\e,j}dx=0$ and $\int_\Om |\D u_{\e,j}|^2dx\leq1$
such that
\begin{align}\label{sub-lim}
\lim_{j\rightarrow\infty}\int_\Om \frac{e^{2\pi(1-\beta-\e)u_{\e,j}^2}}{|x|^{2\beta}}dx
=\sup_{u\in W^{1,2}(\Omega),\int_\Om udx=0, \int_\Om |\D u|^2dx\leq1} \int_{\Omega} \frac{e^{2\pi(1-\beta-\epsilon) u^2}}{|x|^{2\beta}} dx.
\end{align}
Since $u_{\e,j}$ is bounded in $W^{1,2}(\Om)$, there exists some $u_\e\in W^{1,2}(\Om)$ such that up to a subsequence, we have
\begin{align}\label{weak-con}
\begin{cases}
u_{\e,j}\rightharpoonup u_\e \ \ \ \text{weakly \ in}\ \ W^{1,2}(\Om),\\
u_{\e,j}\rightarrow u_\e \ \ \ \text{strongly \ in}\ \ L^p(\Om)\ (\forall p\geq1),\\
u_{\e,j}\rightarrow u_\e \ \ \ \text{almost \ everywhere \ in }\ \ \Om,
\end{cases}
\end{align}
as $j\to\infty$.
Choosing $q>1$ and $s>1$ sufficiently close to $1$ such that $(1-\beta-\e)q+\beta qs=1$, one has by Cauchy-Schwartz inequality that
\begin{align}\label{ineq-Lp}
    \int_\Om \frac{e^{2\pi(1-\beta-\e)qu_{\e,j}^2}}{|x|^{2\beta q}}dx
\leq\left(\int_\Om e^{2\pi u_{\e,j}^2}dx\right)^{(1-\beta-\e)q}\left(\int_\Om \frac{1}{|x|^{\frac{2}{s}}}dx\right)^{\beta qs}.
\end{align}
By (\ref{ineq-cy}) one knows
\begin{align}\label{ineq-Lp2}
\int_\Om e^{2\pi u_{\e,j}^2}dx\leq\sup_{u\in W^{1,2}(\Om),\int_\Om u dx=0,\int_\Om |\D u|^2dx\leq1}\int_\Om e^{2\pi u^2}dx < +\infty.
\end{align}
Combining (\ref{ineq-Lp}) and (\ref{ineq-Lp2}) we know
$|x|^{-2\beta}e^{2\pi(1-\beta-\e)u_{\e,j}^2}$
is bounded in $L^q(\Om)$ for some $q>1$. Since
\begin{align*}
    &|x|^{-2\beta}|e^{2\pi(1-\beta-\e)u_{\e,j}^2}-e^{2\pi(1-\beta-\e)u_{\e}^2}|\nonumber\\
\leq&2\pi(1-\beta-\e)|x|^{-2\beta}(e^{2\pi(1-\beta-\e)u_{\e,j}^2}+e^{2\pi(1-\beta-\e)u_{\e}^2})|u_{\e,j}^2-u_{\e}^2|
\end{align*}
and $u_{\e,j}\rightarrow u_\e$ strongly in $L^p(\Om)$ ($\forall p\geq1$) as $j\rightarrow\infty$, we have by Lebesgue's dominated convergence theorem that
\begin{align}\label{eq-sub2}
 \lim_{j\rightarrow\infty}\int_\Om |x|^{-2\beta}e^{2\pi(1-\beta-\e)u_{\e,j}^2}dx
=\int_\Om |x|^{-2\beta}e^{2\pi(1-\beta-\e)u_{\e}^2}dx.
\end{align}
From (\ref{weak-con}) one has $\int_\Om u_\e dx=0$ and
\begin{align*}
\int_\Om |\D u_\e|^2dx\leq \liminf_{j\rightarrow\infty} \int_\Om |\D u_{\e,j}|^2 dx\leq1.
\end{align*}
By (\ref{sub-lim}) and (\ref{eq-sub2}) we know $u_\e$ attains the supremum in (\ref{eq-sub}), so $\int_\Om |\D u_\e|^2dx\neq0$. Or else, it follows from
Poincar\'{e} inequality that $u_\e\equiv0$, which contradicts (\ref{eq-sub}). In fact one has $\int_\Om |\D u_\e|^2dx=1$.
Suppose not, then $\int_\Om |\D u_\e|^2dx<1$ and
\begin{align*}
\int_\Om |x|^{-2\beta}e^{2\pi(1-\beta-\e)u_{\e}^2}dx<\int_\Om |x|^{-2\beta}e^{2\pi(1-\beta-\e)\left(\frac{u_{\e}}{\sqrt{\int_\Om |\D u_\e|^2dx}}\right)^2}dx,
\end{align*}
which is a contradiction with $u_\e$ attains the supremum. This ends the proof of the lemma.
\end{proof}

\begin{lem}
It holds that
\begin{align}\label{func-lim}
\lim_{\e\rightarrow0}\int_{\Omega} \frac{e^{2\pi(1-\beta-\epsilon) u_{\epsilon}^2}}{|x|^{2\beta}} dx
=\sup_{u\in W^{1,2}(\Omega),\int_\Om udx=0, \int_\Om |\D u|^2dx\leq1} \int_{\Omega} \frac{e^{2\pi(1-\beta) u^2}}{|x|^{2\beta}} dx.
\end{align}
\end{lem}
\begin{proof}
By Lemma \ref{sub-mt}, we know $\int_{\Omega}|x|^{-2\beta}e^{2\pi(1-\beta-\e)u_\e^2}dx$ is increasing when $\e$ tends to $0$. So the limit on the right-hand side of
(\ref{func-lim}) is meaningful. For any $u\in W^{1,2}(\Omega)$ with $\int_\Omega udx=0$ and $\int_\Omega|\nabla u|^2dx\leq1$, one has
\begin{align*}
\int_{\Omega} \frac{e^{2\pi(1-\beta) u^2}}{|x|^{2\beta}} dx
=&\lim_{\e\rightarrow0}\int_{\Omega} \frac{e^{2\pi(1-\beta-\e) u^2}}{|x|^{2\beta}} dx\\
\leq&\lim_{\e\rightarrow0}\sup_{u\in W^{1,2}(\Omega),\int_\Om udx=0, \int_\Om |\D u|^2dx\leq1} \int_{\Omega} \frac{e^{2\pi(1-\beta-\epsilon) u^2}}{|x|^{2\beta}} dx\\
=&\lim_{\e\rightarrow0}\int_{\Omega} \frac{e^{2\pi(1-\beta-\e) u_\e^2}}{|x|^{2\beta}} dx.
\end{align*}
Thus we have
\begin{align}\label{proof-1}
\sup_{u\in W^{1,2}(\Omega),\int_\Om udx=0, \int_\Om |\D u|^2dx\leq1} \int_{\Omega} \frac{e^{2\pi(1-\beta) u^2}}{|x|^{2\beta}} dx
\leq\lim_{\e\rightarrow0}\int_{\Omega} \frac{e^{2\pi(1-\beta-\e) u_\e^2}}{|x|^{2\beta}} dx.
\end{align}
On the other hand, $u_\e\in W^{1,2}(\Omega)$ satisfies $\int_\Omega u_\e dx=0$ and $\int_\Omega |\nabla u_\e|^2dx=1$, we have
\begin{align}\label{proof-2}
\lim_{\e\rightarrow0}\int_{\Omega} \frac{e^{2\pi(1-\beta-\e) u_\e^2}}{|x|^{2\beta}} dx
\leq \sup_{u\in W^{1,2}(\Omega),\int_\Om udx=0, \int_\Om |\D u|^2dx\leq1} \int_{\Omega} \frac{e^{2\pi(1-\beta) u^2}}{|x|^{2\beta}} dx.
\end{align}
Then by combining (\ref{proof-1}) and (\ref{proof-2}) we obtain (\ref{func-lim}). This finishes the proof.
\end{proof}

From Lemma \ref{sub-mt} one knows, $u_\e$ attains the supremum of the the subcritical singular Moser-Trudinger functional $\int_\Om |x|^{-2\beta}e^{2\pi(1-\beta-\e)u^2}dx$
on function space $$\left\{u\in W^{1,2}(\Om): \int_\Om udx=0, \int_\Om |\D u|^2dx\leq1\right\},$$
so it satisfies the following Euler-Lagrange equation
\begin{align}\label{eq-EL}
\begin{cases}
-\Delta u_\e=\lambda_\e^{-1}(f_\e-\overline{f_\e}) \ \ &\text{in}\ \ \Omega,\\
\frac{\p u_\e}{\p\nu}=0\ \ \ &\text{on}\ \ \p\Omega,
\end{cases}
\end{align}
where $\nu$ is the unit outward normal vector of $\p\Om$, $f_\e=|x|^{-2\beta}u_\e e^{2\pi(1-\beta-\e)u_\e^2}$, $\overline{f_\e}$ is the mean value
of $f_\e$ on $\Om$ and $\lambda_\e=\int_\Om u_\e f_\e dx$.

\begin{lem}\label{2.3}
It holds that
\begin{align}\label{lowbound-lam}
\liminf_{\e\rightarrow0}\lambda_\e>0  \ \ \ \ \text{and} \ \ \ \ \lambda_\e^{-1}|\overline{f_\e}|\leq C.
\end{align}
\end{lem}

\begin{proof}
It follows by the simple $e^t\leq1+te^t \ (\forall t\geq0)$ that
\begin{align*}
\int_\Om \frac{e^{2\pi(1-\beta-\e)u_\e^2}}{|x|^{2\beta}}dx\leq\int_\Om \frac{1}{|x|^{2\beta}}dx+2\pi(1-\beta-\e)\lambda_\e.
\end{align*}
This together with (\ref{eq-sub}) tells us that
\begin{align}\label{proof-3}
\liminf_{\e\rightarrow0}\lambda_\e>0.
\end{align}
To see the second inequality in (\ref{lowbound-lam}), we have
\begin{align*}
\lambda_\e^{-1}|\overline{f_\e}|
\leq& \frac{\lambda_\e^{-1}}{|\Om|}\int_{|u_\e|\leq1}\frac{1}{|x|^{2\beta}}|u_\e|e^{2\pi(1-\beta-\e)u_\e^2}dx
    + \frac{\lambda_\e^{-1}}{|\Om|}\int_{|u_\e|>1}\frac{1}{|x|^{2\beta}}|u_\e|e^{2\pi(1-\beta-\e)u_\e^2}dx\\
\leq&  \frac{\lambda_\e^{-1}}{|\Om|}\int_{\Om}\frac{1}{|x|^{2\beta}}e^{2\pi(1-\beta)}dx
    + \frac{\lambda_\e^{-1}}{|\Om|}\int_{\Om}\frac{1}{|x|^{2\beta}}u_\e^2e^{2\pi(1-\beta-\e)u_\e^2}dx \\
\leq& C+\frac{1}{|\Om|},
\end{align*}
where in the last inequality we have used (\ref{proof-3}). This ends the proof.
\end{proof}

In view of Lemma \ref{2.3}, we have by applying elliptic estimates to (\ref{eq-EL}) that
\begin{align}\label{regularity}
u_\epsilon\in W^{1,2}(\Omega)\cap C^1_{\text{loc}}(\overline \Omega\setminus \{0\})\cap C^0(\overline \Omega).
\end{align}

\subsection{Blow-up analysis}

Since $u_\e$ is bounded in $W^{1,2}(\Om)$, there exists some $u_0\in W^{1,2}(\Om)$ such that up to a subsequence that
\begin{align}\label{weak-con2}
\begin{cases}
u_{\e}\rightharpoonup u_0 \ \ \ \text{weakly \ in}\ \ W^{1,2}(\Om),\\
u_{\e}\rightarrow u_0 \ \ \ \text{strongly \ in}\ \ L^p(\Om)\ (\forall p\geq1),\\
u_{\e}\rightarrow u_0 \ \ \ \text{almost \ everywhere \ in }\ \ \Om,
\end{cases}
\end{align}
as $\e\to0$.
We denote $c_\e=\max_{\overline \Om}|u_\e|$. If $c_\e$ is bounded, then for any $u\in W^{1,2}(\Om)$ with $\int_\Om udx=0$ and $\int_\Om |\D u|^2dx\leq1$, one has
by the Lebesgue dominated convergence theorem that
\begin{align*}
\int_{\Om}\frac{e^{2\pi(1-\beta)u^2}}{|x|^{2\beta}}dx
=\lim_{\e\rightarrow0}\int_{\Om}\frac{e^{2\pi(1-\beta-\e)u^2}}{|x|^{2\beta}}dx
\leq\lim_{\e\rightarrow0}\int_{\Om}\frac{e^{2\pi(1-\beta-\e)u_\e^2}}{|x|^{2\beta}}dx
=\lim_{\e\rightarrow0}\int_{\Om}\frac{e^{2\pi(1-\beta)u_0^2}}{|x|^{2\beta}}dx.
\end{align*}
Then $u_0$ attains the supremum and we are done. Therefore, we assume $c_\e\rightarrow+\infty$ as $\e\rightarrow0$ in the sequel. We can assume $c_\e=u_\e(x_\e)$(or else,
one can use $-u_\e$ instead $u_\e$) for some $x_\e\in\overline\Om$ and $x_\e\rightarrow x_0\in\overline\Om$
as $\e\rightarrow0$.

\begin{lem}\label{uzero} We have $u_0\equiv0$, $x_0=0$, and $|\D u_\e|^2dx\rightharpoonup\delta_0$ as $\e\rightarrow0$ in the sense of measure,
where $\delta_0$ denotes the usual Dirac measure centered at the origin $0$.
\end{lem}

\begin{proof} It follows by (\ref{weak-con2}) that
$$\int_\Om u_0 dx=\lim_{\e\con0}\int_\Om u_\e dx =0.$$
If $u_0\not\equiv0$, then from Poincar\'{e} inequality we know $\int_\Om |\D u_0|^2dx>0$. Thus
\begin{align*}
\int_\Om |\D(u_\e-u_0)|^2 dx=1-\int_\Om |\D u_0|^2 dx+o_\e(1)<1-\frac{1}{2}\int_\Om |\D u_0|^2 dx
\end{align*}
for $\e>0$ sufficiently small. For any $q\in(1,1/\beta)$, $\delta>0$, $s>1$ and $s'=s/(s-1)$, using Cauchy's inequality and H\"{o}lder's inequality one has
\begin{align*}
\int_\Om \frac{e^{2\pi(1-\beta-\e)qu_\e^2}}{|x|^{2\beta q}}dx
\leq \left(\int_\Om \frac{e^{2\pi(1-\beta-\e)q(1+\delta)s(u_\e-u_0)^2}}{|x|^{2\beta q}}dx\right)^{1/s}
     \left(\int_\Om \frac{e^{2\pi(1-\beta-\e)q(1+\frac{1}{4\delta})s'u_0^2}}{|x|^{2\beta q}}dx\right)^{1/s'}.
\end{align*}
By choosing $q$, $1+\delta$ and $s$ sufficiently close to $1$ such that
$$(1-\beta)q(1+\delta)s(1-\frac{1}{2}\int_\Om |\nabla u_0|^2 dx)+\beta q < 1,$$
we know $|x|^{-2\beta}e^{2\pi(1-\beta-\e)u_\e^2}$ is bounded in $L^q(\Om)$ for some $q>1$. Since $u_\e$ is bounded in $L^p(\Om)$ for any $p\geq1$, we have
by H\"{o}lder's inequality that $|x|^{-2\beta}u_\e e^{2\pi(1-\beta)u_\e^2}=f_\e$ is bounded in $L^r(\Om)$ for some $1<r<q$. This together with
Lemma \ref{2.3} shows the right-hand side of equation (\ref{eq-EL}) is bounded in $L^r(\Om)$. The standard elliptic estimates tell us that
 $u_\e$ is bounded in $C^0(\overline\Om)$. It contradicts $c_\e\rightarrow+\infty$ as $\e\rightarrow0$. Hence $u_0\equiv0$.

Since $\int_\Om |\D u_\e|^2 dx=1$, we have $|\D u_\e|^2dx\rightharpoonup\delta_{x_0}$ as $\e\rightarrow0$. Or else, one can choose a sufficiently
small $r_0>0$ and a cut-off function $\eta\in C_0^1(B_{4r_0}(x_0)\cap\Om)$ with $0\leq\eta\leq1$ in $B_{4r_0}(x_0)\cap\Om$, $\eta\equiv1$ in $B_{2r_0}(x_0)\cap\Om$
and $|\nabla\eta|\leq C/r_0$ such that
\begin{align}\label{proof-4}
\limsup_{\e\con0}\int_{\Om}|\nabla(\eta u_\e)|^2dx=\limsup_{\e\con0}\int_{B_{4r_0}(x_0)\cap\Om}|\nabla(\eta u_\e)|^2dx<1.
\end{align}
The fact $u_0\equiv0$ and (\ref{weak-con2}) tells us that
\begin{align}\label{proof-5}
\lim_{\e\con0}\int_\Om \eta u_\e dx = \int_\Om \eta u_0 = 0.
\end{align}
Then by (\ref{func-sub}), (\ref{proof-4}) and (\ref{proof-5}) we have $|x|^{-2\beta}e^{2\pi(1-\beta-\e)(\eta u_\e)^2}$ is bounded in $L^s(\Om)$, for some $s>1$.
From (\ref{weak-con}) one knows $u_\e$ is bounded in $L^q(\Om)$ for any $q\geq1$. Then it follows by H\"{o}lder's inequality that
$|x|^{-2\beta}u_\e e^{2\pi(1-\beta-\e)(\eta u_\e)^2}$ is bounded in $L^p(\Om)$ for some $p>1$. Since $\eta\equiv1$ in $B_{2r_0}(x_0)\cap\Om$,
we have by using elliptic estimates to equation (\ref{eq-EL}) that $u_\e$ is uniformly bounded in $B_{r_0}(x_0)\cap\Om$. It contradicts
the assumption $c_\e\con+\infty$ as $\e\con0$. Then we have $|\D u_\e|^2dx\rightharpoonup\delta_{x_0}$ as $\e\rightarrow0$ in the sense of measure.

Suppose $x_0\neq0$. Now $\lambda_\e^{-1}|x|^{-2\beta}u_\e e^{2\pi(1-\beta-\e)u_\e^2}$ is bounded in $L^{p_1}(B_{|x_0|/2}\cap\Om)$ for some $p_1>1$.
When $|x|\geq|x_0|/2$, we have $|x|^{-2\beta}\leq(|x_0|/2)^{-2\beta}$ so $\lambda_\e^{-1}|x|^{-2\beta}u_\e e^{2\pi(1-\beta-\e)u_\e^2}$ is bounded in
$L^{p_2}(\Om\setminus B_{|x_0|/2})$ for some $p_2>1$. Concluding we have $\lambda_\e^{-1}|x|^{-2\beta}u_\e e^{2\pi(1-\beta-\e)u_\e^2}$ is bounded in
$L^{p_3}(\Om)$ for some $p_3=\min\{p_1,p_2\}>1$. By the standard elliptic estimates, $c_\e$ is bounded. This contradicts $c_\e\rightarrow+\infty$
as $\e\rightarrow0$. Then we have $x_0=0$ and finish the proof.
\end{proof}

We define
\begin{align*}
r_\e=\lambda_\e^{1/2}c_\e^{-1} e^{-\pi(1-\beta-\e)c_\e^2}.
\end{align*}
\begin{lem}\label{2.5}
For any $\tau\in(0,1-\beta)$, there holds
\begin{align}\label{scaling-to-0}
r_\e^2 e^{2\pi\tau c_\e^2}\rightarrow0 \ \ \text{as}\ \ \e\rightarrow0.
\end{align}
\end{lem}

\begin{proof}
For any $\tau\in(0,1-\beta)$, we have by the definition of $r_\e$ that
\begin{align}\label{pf-scaling}
r_\e^2 e^{2\pi\tau c_\e^2} &=\lambda_\e c_\e^{-2}e^{-2\pi(1-\beta-\e)c_\e^2}e^{2\tau c_\e^2}\nonumber\\
                           &=c_\e^{-2}e^{-2\pi(1-\beta-\e-\tau)c_\e^2}\int_\Om |x|^{-2\beta}u_\e^2 e^{2\pi(1-\beta-\e)u_\e^2}dx\nonumber\\
                           &\leq c_\e^{-2}\int_\Om |x|^{-2\beta}u_\e^2 e^{2\pi\tau u_\e^2}dx
\end{align}
for sufficiently small $\e>0$. Since $|x|^{-2\beta}e^{2\pi\tau u_\e^2}$ is bounded in $L^q(\Om)$ for some $q>1$ and $u_\e\rightarrow0$ strongly in $L^p(\Om)$
for any $p\geq1$, (\ref{scaling-to-0}) follows by (\ref{pf-scaling}) directly.
\end{proof}

Since $\Om$ is smooth, we can assume without loss of generality that $\p\Om$ is flat near the origin, i.e. there exists a $\delta>0$, such that
$B_\delta^+(0):=\{(x^1,x^2)\in\mathbb{R}_+^2|(x^1)^2+(x^2)^2<\delta\}\subset\Om$ and $\p B_\delta^+(0)\cap\p\mathbb{R}^2_{+}\subset\p\Om$,
where $\mathbb{R}^2_+=\{(x^1,x^2)\in\mathbb{R}^2: x^2\geq0\}$.

Before beginning blow-up analysis, we reflect $u_\e$ to $B_\delta(0)\setminus B_\delta^+(0)$ first. Since $\frac{\p u_\e}{\p \nu}|_{\p\Om}=0$, we can define
\begin{align*}
\widetilde{u}_\e(x):=
\begin{cases}
u_\e(x),& x\in B_\de^+(0),\\
u_\e(x^1,-x^2),& x\in B_\delta(0)\setminus B_\delta^+(0).
\end{cases}
\end{align*}
Denoting $t_\e=r_\e^{1/(1-\beta)}$ and $\Om_{2,\e}=\{x\in\mathbb{R}^2:~x_\e+t_\e x\in B_\de(0)\}$, we define on $\Om_{2,\e}$ two blow-up sequences as follows:
\begin{align}\label{blowup-functions}
\psi_\e(x)=c_\e^{-1}\widetilde{u}_\e(x_\e+t_\e x),\ \ \ \ \varphi_\e(x)=c_\e(\widetilde{u}_\e(x_\e+t_\e x)-c_\e).
\end{align}

\begin{lem}\label{bubble}
Let $\psi_\e$ and $\varphi_\e$ be defined as (\ref{blowup-functions}).
We have
$\psi_\e\con 1$ and $\varphi_\e\to \varphi_0$ in
$W^{1,2}_{\emph{\text{loc}}}(\mathbb{R}^2)\cap C^1_{\emph{\text{loc}}}(\mathbb{R}^2\setminus\{0\})\cap C^0_{\emph{\text{loc}}}(\mathbb{R}^2)$
as $\e\con0$, where
\begin{align}\label{eq-varphi0}
\varphi_0(x)=-\frac{1}{2\pi(1-\beta)}\log\left(1+\frac{\pi}{2(1-\beta)}|x|^{2(1-\beta)}\right).
\end{align}
\end{lem}

\begin{proof}
First, we claim that there exists a constant $C$ such that
\begin{align}\label{claim}
t_\e^{-1}|x_\e|\leq C.
\end{align}
Suppose (\ref{claim}) does not hold, then one has
\begin{align}\label{claim-pf-1}
t_\e^{-1}|x_\e|\con+\infty \ \ \  \text{as}\ \ \ \e\con0.
\end{align}
Let $s_\e=t_\e^{1-\beta}|x_\e|^{\beta}$, then $s_\e\con0$ as $\e\con0$ follows from Lemma \ref{2.5}.
We define two blow-up sequences on
$\Om_{1,\e}=\{x\in\mathbb{R}^2:~x_\e+s_\e x\in B_\de(0)\}$ as follows:
\begin{align}\label{vw-eps}
v_\e(x)=c_\e^{-1}\widetilde{u}_\e(x_\e+s_\e x),\ \ \ w_\e(x)=c_\e(\widetilde{u}_\e(x_\e+s_\e x)-c_\e).
\end{align}
By equation (\ref{eq-EL}) one obtains in $\Om_{1,\e}$ that
\begin{align*}
-\Delta v_\e &= c_\e^{-2}|x_\e|^{2\beta}|x_\e+s_\e x|^{-2\beta}v_\e e^{2\pi(1-\beta-\e)c_\e^2(v_\e^2-1)}-c_\e^{-1}s_\e^2\lambda_\e^{-1}\overline{f_\e},\\
-\Delta w_\e &= |x_\e|^{2\beta}|x_\e+s_\e x|^{-2\beta}v_\e e^{2\pi(1-\beta-\e)(v_\e+1)w_\e}-c_\e s_\e^2\lambda_\e^{-1}\overline{f_\e}.
\end{align*}
Lemmas \ref{2.3} and \ref{2.5} tell us that $c_\e^{-1}s_\e^2\lambda_\e^{-1}\overline{f_\e}=o_\e(1)$, $c_\e s_\e^2\lambda_\e^{-1}\overline{f_\e}=o_\e(1)$.
Noticing that $v_\e\leq1$ and locally $|x_\e|^{2\beta}|x_\e+s_\e x|^{-2\beta}=1+o_\e(1)$. Then we know by elliptic estimates that $v_\e\con v_0$ in $C^1_{\text{loc}}(\mathbb{R}^2)$ as
$\e\con0$, where $v_0$ satisfies
$$-\Delta v_0(x) = 0,\ \ x\in\mathbb{R}^2.$$
Since $v_0(x)\leq\lim_{\e\con0}v_\e(x)\leq1$ and $v_0(0)=\lim_{\e\con0}v_\e(0)=1$, the Liouville theorem tells us that $v_0\equiv1$.
By elliptic estimates we also know
$w_\e\to w_0$
in $C_{\text{loc}}^1(\mathbb{R}^2)$, where $w_0$ satisfies
\begin{align*}
\begin{cases}
-\Delta w_0 = e^{4\pi(1-\beta)w_0}\ \ \text{in} \ \ \mathbb{R}^2\\
w_0(0)=\sup_{\mathbb{R}^2}w_0=0.
\end{cases}
\end{align*}
Recalling the definition of $\widetilde{u}_\e$ one has
\begin{align}\label{upbound2}
\int_{\mathbb{R}^2}e^{4\pi(1-\beta)w_0(x)}dx
=\lim_{R\to+\infty}\lim_{\e\to0}\lambda_\e^{-1}\int_{B_{Rs_\e}(x_\e)}|y|^{-2\beta}\widetilde{u}_\e^2e^{2\pi(1-\beta-\e)\widetilde{u}_\e^2}dy\leq2.
\end{align}
Then by the classification theorem of Chen-Li (\cite{CL}, Theorem 1) we have
$$w_0(x)=-\frac{1}{2\pi(1-\beta)}\log(1+\frac{\pi(1-\beta)}{2}|x|^2)$$
and then
\begin{align}\label{lowbound2}
\int_{\mathbb{R}^2}e^{4\pi(1-\beta)w_0(x)}dx=\frac{2}{1-\beta}>2.
\end{align}
The contradiction between (\ref{upbound2}) and (\ref{lowbound2}) shows the claim (\ref{claim}) holds.

By equation (\ref{eq-EL}) and a direct calculation, we have in $\Om_{2,\e}$ that
\begin{align*}
-\Delta \psi_\e &= c_\e^{-2}|t_\e^{-1}x_\e+ x|^{-2\beta}\psi_\e e^{2\pi(1-\beta-\e)c_\e^2(\psi_\e^2-1)}-c_\e^{-1}t_\e^2\lambda_\e^{-1}\overline{f_\e},\\
-\Delta \varphi_\e &= |t_\e^{-1}x_\e+ x|^{-2\beta}\psi_\e e^{2\pi(1-\beta-\e)(\psi_\e+1)\varphi_\e}-c_\e t_\e^2\lambda_\e^{-1}\overline{f_\e}.
\end{align*}
In view of (\ref{claim}), there exists a $x_*$ such that $t_\e^{-1}x_\e\to x_*$
as $\e\to0$.
By Lemmas \ref{2.3} and \ref{2.5} one knows, $c_\e^{-1}t_\e^2\lambda_\e^{-1}\overline{f_\e}=o_\e(1)$ and $c_\e t_\e^2\lambda_\e^{-1}\overline{f_\e}=o_\e(1)$.
Noticing also that $\psi_\e\leq1$ and $|t_\e^{-1}x_\e+ x|^{-2\beta}$ is bounded in $L^p_{\text{loc}}(\mathbb{R}^2)$ for some $p>1$,
we have by elliptic estimates that
$\psi_\e\con \psi_0$ in
 $C^1_{\text{loc}}(\mathbb{R}^2\setminus\{-x_*\})\cap C^0_{\text{loc}}(\mathbb{R}^2)$ as
$\e\con0$, where $\psi_0$ satisfies
$$-\Delta \psi_0(x) = 0,\ \ x\in\mathbb{R}^2.$$
Since $\psi_0(x)\leq\lim_{\e\con0}\psi_\e(x)\leq1$ and $\psi_0(0)=\lim_{\e\con0}\psi_\e(0)=1$, the Liouville theorem tells us that $\psi_0\equiv1$.
By elliptic estimates we have that $\varphi_\e\to \varphi_0$
in $W^{1,2}_{\text{loc}}(\mathbb{R}^2)\cap C^1_{\text{loc}}(\mathbb{R}^2\setminus\{-x_*\})\cap C^0_{\text{loc}}(\mathbb{R}^2)$ as
$\e\con0$, where $\varphi_0$ satisfies
\begin{align*}
-\Delta \varphi_0(x) = |x_*+x|^{-2\beta}e^{4\pi(1-\beta)\varphi_0(x)},\ \ \ \ x\in\mathbb{R}^2.
\end{align*}
Still one has
\begin{align*}
    &\int_{\mathbb{R}^2}|x_*+x|^{-2\beta}e^{4\pi(1-\beta)\varphi_0(x)}dx\\
\leq&\lim_{R\to+\infty}\lim_{\e\to0}\lambda_\e^{-1}\int_{B_{Rt_\e}(x_\e-t_\e x_*)}|y|^{-2\beta}\widetilde{u}_\e^2e^{2\pi(1-\beta-\e)\widetilde{u}_\e^2}dy\nonumber\\
=&\lim_{R\to+\infty}\lim_{\e\to0}\int_{B_R(-x_*)}|t_\e^{-1} x_\e+x|^{-2\beta}\psi_\e^2(x)e^{2\pi(1-\beta-\e)(\psi_\e(x)+1)\varphi_\e(x)}dx\nonumber\\
=&2.\nonumber
\end{align*}
By the classification theorem of Chen-Li (\cite{CL2}, Theorem 3.1) or Prajapat-Tarantello (\cite{PT}, Theorem 1.1), one has
\begin{align*}
\varphi_0(x)=-\frac{1}{2\pi(1-\beta)}\log\left(1+\frac{\pi}{2(1-\beta)}|x_*+x|^{2(1-\beta)}\right).
\end{align*}
Noticing $\varphi_0(0)=\lim_{\e\to0}\varphi_\e(0)=0$, we have $x_*=0$ and then
\begin{align*}
\varphi_0(x)=-\frac{1}{2\pi(1-\beta)}\log\left(1+\frac{\pi}{2(1-\beta)}|x|^{2(1-\beta)}\right).
\end{align*}
It follows that
\begin{align}\label{integral}
\int_{\mathbb{R}^2}|x|^{-2\beta}e^{4\pi(1-\beta)\varphi_0(x)}dx=2.
\end{align}
This completes the proof of Lemma \ref{bubble}.
\end{proof}

\subsection{Upper bound estimate}
Similar as Li \cite{Li}, we define $u_{\e,\gamma}=\min\{\gamma c_\e, u_\e\}$.
\begin{lem}\label{lem-gamma}
For any $\gamma\in(0,1)$, there holds that
\begin{align*}
\lim_{\e\to0}\int_\Om |\D u_{\e,\gamma}|^2 dx = \gamma.
\end{align*}
\end{lem}
\begin{proof}
In view of (\ref{eq-EL}), integrating by parts we have
\begin{align*}
   \int_\Om |\D u_{\e,\gamma}|^2 dx
=& \int_\Om \D u_{\e,\gamma} \D u_\e dx = -\int_\Om u_{\e,\gamma} \Delta u_\e dx\nonumber\\
=& \lambda_\e^{-1}\int_\Om |y|^{-2\beta}u_{\e,\gamma}u_\e e^{2\pi(1-\beta-\e)u_\e^2}dy - \lambda_\e^{-1}\overline{f_\e}\int_\Om u_{\e,\gamma}dx\nonumber\\
=& \lambda_\e^{-1}\int_\Om |y|^{-2\beta}u_{\e,\gamma}u_\e e^{2\pi(1-\beta-\e)u_\e^2}dy + o_\e(1)\nonumber\\
\geq& \lambda_\e^{-1}\int_{B^+_{Rt_\e}(x_\e)} |y|^{-2\beta}u_{\e,\gamma}u_\e e^{2\pi(1-\beta-\e)u_\e^2}dy + o_\e(1)\nonumber\\
=& \lambda_\e^{-1}\int_{B^+_{R}(0)} |x_\e+t_\e x|^{-2\beta}\gamma (1+o_\e(1))\psi^2_\e(x)c_\e^2 e^{2\pi(1-\beta-\e)u_\e^2}t_\e^2dx + o_\e(1)\nonumber\\
=& \gamma (1+o_\e(1))\int_{B^+_{R}(0)}|x|^{-2\beta}e^{4\pi(1-\beta)\varphi_0(x)}dx\nonumber\\
=& \gamma (1+o_\e(1))(1+o_R(1)).
\end{align*}
Letting $\e\to0$ first and then $R\to+\infty$ one obtains that
\begin{align}\label{gamma-low}
\liminf_{\e\to0}\int_\Om |\D u_{\e,\gamma}|^2 dx \geq \gamma.
\end{align}
Noticing that $|\D (u_\e-\gamma c_\e)^+|^2 = \D (u_\e-\gamma c_\e)^+\D u_\e$ and $(u_\e-\gamma c_\e)^+ = (1+o_\e(1))(1-\gamma)c_\e$ on
$B^+_{Rt_\e}(x_\e)$, we have similarly as above that
\begin{align}\label{gamma-up}
\liminf_{\e\to0}\int_\Om |\D (u_\e-\gamma c_\e)^+|^2 dx \geq 1-\gamma.
\end{align}
Since $|\D u_\e|^2=|\D u_{\e,\gamma}|^2+|\D (u_\e-\gamma c_\e)^+|^2$,
\begin{align*}
\int_\Om |\D u_{\e,\gamma}|^2 dx + \int_\Om |\D (u_\e-\gamma c_\e)^+|^2 dx = \int_\Om |\D u_\e|^2 dx =1.
\end{align*}
This together with (\ref{gamma-low}) and (\ref{gamma-up}) completes the proof of the lemma.
\end{proof}

As an application of Lemma \ref{lem-gamma}, one has
\begin{cor}\label{cor}
There holds
\begin{align*}
\lim_{\e\to0}\int_\Om \frac{e^{2\pi(1-\beta-\e)u_\e^2}}{|x|^{2\beta}}dx \leq \int_\Om \frac{1}{|x|^{2\beta}}dx+\limsup_{\e\to0}\frac{\lambda_\e}{c_\e^2}.
\end{align*}
\end{cor}
\begin{proof}
For any $\gamma\in(0,1)$, we have
\begin{align}\label{cor-proof1}
 \int_\Om \frac{e^{2\pi(1-\beta-\e)u_\e^2}}{|x|^{2\beta}}dx
=& \int_{u_\e\leq\gamma c_\e} \frac{e^{2\pi(1-\beta-\e)u_\e^2}}{|x|^{2\beta}}dx + \int_{u_\e>\gamma c_\e}
    \frac{e^{2\pi(1-\beta-\e)u_\e^2}}{|x|^{2\beta}}dx\\
\leq& \int_{\Om} \frac{e^{2\pi(1-\beta-\e)u_{\e,\gamma}^2}}{|x|^{2\beta}}dx + \frac{\lambda_\e}{\gamma^2 c_\e^2}.\nonumber
\end{align}
From Lemma \ref{lem-gamma} we know $\frac{e^{2\pi(1-\beta-\e)u_{\e,\gamma}^2}}{|x|^{2\beta}}$ is bounded in $L^q(\Om)$ for some $q>1$. Noticing that
$u_{\e,\gamma}$ converges to $0$ almost everywhere in $\Om$. Therefore, $\frac{e^{2\pi(1-\beta-\e)u_{\e,\gamma}^2}}{|x|^{2\beta}}$ converges to $\frac{1}{|x|^{2\beta}}$ in $L^1(\Om)$. We finish the proof by letting $\e\to0$ first and then $\gamma\to1$ in (\ref{cor-proof1}).
\end{proof}

From this corollary we can obtain that
\begin{align}\label{corcor}
\limsup_{\e\to0}\frac{\lambda_\e}{c_\e^\theta}=+\infty,\ \ \ \ \forall \theta\in(0,2).
\end{align}
In fact, if (\ref{corcor}) does not hold, then one has $\lambda_\e/c_\e^2\to0$ as $\e\to0$. Let $v\in W^{1,2}(\Om)$ with $\int_\Om v dx=0$ and $||\D v||_2=1$.
It follows by Corollary \ref{cor} that
\begin{align*}
\int_\Om |x|^{2\beta}e^{2\pi(1-\beta)v^2}dx
\leq& \sup_{u\in W^{1,2}(\Om), \int_\Om udx=0, \int_\Om |\D u|^2dx\leq1} \int_\Om |x|^{-2\beta}e^{2\pi(1-\beta)u^2}dx\nonumber\\
=& \lim_{\e\to0}\int_\Om |x|^{-2\beta}e^{2\pi(1-\beta-\e)u_\e^2}dx\nonumber\\
=&\int_\Om |x|^{-2\beta}dx.
\end{align*}
This is impossible since $v\neq0$. Thus (\ref{corcor}) holds.

\begin{lem}\label{lem-green}
For any $q\in(1,2)$, $c_\e u_\e$ is bounded in $W^{1,q}(\Om)$ and converges weakly to the Green function $G$ which satisfies
\begin{align}\label{eq-green}
\begin{cases}
-\Delta G = \delta_0-\frac{1}{|\Om|} \ \ &\text{in}\ \ \overline{\Om},\\
\frac{\p G}{\p \nu}=0 \ \ \ \ &\text{on}\ \ \p\Om\setminus\{0\},\\
\int_\Om G dx=0.
\end{cases}
\end{align}
Furthermore, $c_\e u_\e\to G$ in $C_{\emph{\text{loc}}}^1(\overline{\Om}\setminus\{0\})$ as $\e\con0$.
\end{lem}

\begin{proof}
For any $\phi\in C^0(\overline{\Om})$, we have
\begin{align}\label{phi-0}
\lim_{\e\to0}c_\e\lambda_\e^{-1}\int_\Om f_\e\phi dx = \phi(0).
\end{align}
In fact, one knows from the proof of Lemma \ref{bubble} that $t_\e^{-1}x_\e\to0$ as $\e\to0$. So
 $$B^+_{(R-1)t_\e}(0)\subset\Om\cap B_{Rt_\e}(x_\e)\subset B^+_{(R+1)t_\e}(0).$$
Then
\begin{align}\label{phi-3}
c_\e\lambda_\e^{-1}\int_\Om f_\e\phi dx = &c_\e\lambda_\e^{-1}\int_{\{x\in\Om: u_\e(x)<\gamma c_\e\}} f_\e\phi dx
                       +c_\e\lambda_\e^{-1}\int_{\{x\in\Om: u_\e\geq\gamma c_\e\}\cap B^+_{Rt_\e}(0)} f_\e\phi dx\\
                       &+c_\e\lambda_\e^{-1}\int_{\{x\in\Om: u_\e\geq\gamma c_\e\}\setminus B^+_{Rt_\e}(0)}f_\e\phi dx\nonumber\\
                     := & I_1+I_2+I_3.\nonumber
\end{align}
Now we estimate the integrals on the right hand side of (\ref{phi-3}).
Since $u_\e$ is bounded in $L^p(\Om)$ for any $p\geq1$ and $|x|^{-2\beta}e^{2\pi(1-\beta-\e)u_{\e,\gamma}^2}$
is bounded in $L^q(\Om)$ for some $q>1$ by Lemma \ref{lem-gamma}, combining (\ref{corcor}) we have
\begin{align}\label{phi-4}
|I_1|
\leq c_\e \lambda_\e^{-1}||\phi||_{C^0(\overline{\Om})}\int_{\Om} |x|^{-2\beta}|u_\e| e^{2\pi(1-\beta-\e)u_{\e,\gamma}^2}dx
=o_\e(1).
\end{align}
Since $B^+_{Rt_\e}(0)\subset\{x\in\Om:\ u_\e(x)\geq\gamma c_\e\}$ for sufficiently small $\e>0$, one has by (\ref{integral}) that
\begin{align}\label{phi-5}
 I_2=&\phi(0)(1+o_\e(1))\int_{B^+_{Rt_\e}(0)} \lambda_\e^{-1}|x|^{-2\beta}c_\e u_\e e^{2\pi(1-\beta-\e)u_\e^2} dx\\
    =&\phi(0)(1+o_\e(1))\int_{B^+_{R}(0)} |y|^{-2\beta}\frac{u_\e(t_\e y)}{c_\e} e^{2\pi(1-\beta-\e)(\frac{u_\e(t_\e y)}{c_\e}+1)c_\e(u_\e(t_\e y)-c_\e)} dy.\nonumber
\end{align}
Noticing that
\begin{align*}
\frac{u_\e(t_\e y)}{c_\e}=\psi_\e(-t_\e^{-1}x_\e+y),\  \ c_\e(u_\e(t_\e y)-c_\e)=\varphi_\e(-t_\e^{-1}x_\e+y).
\end{align*}
By Lemma \ref{bubble} and diagonal arguments, we have
\begin{align*}
\psi_\e(-t_\e^{-1}x_\e+y)\to\psi_0(y)\equiv1,\ \ \varphi_\e(-t_\e^{-1}x_\e+y)\to\varphi_0(y) \ \ \text{as}\ \ \e\to0.
\end{align*}
Inserting them in (\ref{phi-5}), one obtains
\begin{align}\label{phi-6}
I_2=&\phi(0)(1+o_\e(1))\int_{B^+_{R}(0)} |y|^{-2\beta}(1+o_\e(1)) e^{4\pi(1-\beta-\e)(1+o_\e(1))\varphi_0(y)}dy\\
   =&\phi(0)(1+o_\e(1))(1+o_R(1))\int_{\mathbb{R}_+^2}|y|^{-2\beta}e^{4\pi(1-\beta)\varphi_0(y)}dy\nonumber\\
   =&\phi(0)(1+o_\e(1))(1+o_R(1)).\nonumber
\end{align}
For $I_3$, one has
\begin{align}\label{phi-7}
|I_3|\leq& ||\phi||_{C^0(\overline{\Om})}\frac{1}{\gamma}\int_{\{x\in\Om:\ u_\e\geq\gamma c_\e\}\setminus B^+_{Rt_\e}(0)}
                   \lambda_\e^{-1}|x|^{-2\beta}u_\e^2 e^{2\pi(1-\beta-\e)u_\e^2} dx\\
     \leq& C\left(1-(1+o_\e(1))\int_{B^+_{R}(0)} |x|^{-2\beta}e^{4\pi(1-\beta)\varphi_0(x)}dx\right) \nonumber\\
     =&o_\e(1)+o_R(1).\nonumber
\end{align}
Then we have (\ref{phi-0}) by inserting (\ref{phi-4}), (\ref{phi-6}) and (\ref{phi-7}) into (\ref{phi-3}).

By equation (\ref{eq-EL}), $c_\e u_\e$ is a distributional solution to
\begin{align}\label{phi-8}
\begin{cases}
-\Delta (c_\e u_\e)=c_\e\lambda_\e^{-1}(f_\e-\overline{f_\e}) \ \ &\text{in}\ \ \Om,\\
\frac{\p (c_\e u_\e)}{\p \nu}=0\ \ &\text{on}\ \ \p\Om.
\end{cases}
\end{align}
From (\ref{phi-0}) we know $c_\e\lambda_\e^{-1}f_\e$ is bounded in $L^1(\Om)$ and $c_\e\lambda_\e^{-1}\overline{f_\e}\to\frac{1}{|\Om|}$ as $\e\to0$.
Similar as Theorem A.2 in \cite{Li}one can show $\D (c_\e u_\e)$ is bounded in $L^q(\Om)$ for any $q\in(1,2)$,
then by Poincar\'{e} inequality we have $c_\e u_\e$ is bounded in $W^{1,q}(\Om)$.
Without loss of generality, one can assume $c_\e u_\e\rightharpoonup G$ weakly in $W^{1,q}(\Om)$. Then for any smooth $\phi$,
\begin{align*}
   \int_\Om \D \phi \D (c_\e u_\e) dx
=& -\int_{\Om} \phi \Delta(c_\e u_\e) dx + \int_{\p \Om} \phi \frac{\p (c_\e u_\e)}{\p\nu}ds\\
=& c_\e\lambda_\e^{-1}\int_{\Om} f_\e \phi dx-c_\e\lambda_\e^{-1}\overline{f_\e}\int_\Om \phi dx\\
\rightarrow& \phi(0)-\frac{1}{|\Om|}\int_\Om \phi dx
\end{align*}
as $\e\to0$. Hence we have
$$\int_\Om \D \phi \D G dx = \phi(0)-\frac{1}{|\Om|}\int_\Om \phi dx.$$
Then (\ref{eq-green}) holds. By elliptic estimates we have $c_\e u_\e\to G$ in $C_{\emph{\text{loc}}}^1(\overline{\Om}\setminus\{0\})$ as $\e\con0$.
This finishes the proof.
\end{proof}

It is well-known that locally the Green function $G$ takes the form
\begin{align}\label{green-exp}
G(x)=-\frac{1}{\pi}\log r + A_0 + \psi(x),
\end{align}
where $A_0$ is a constant and $\psi\in C^1(\overline{\Om})$.\\

We use the capacity method which was first explored by Li \cite{Li} to derive the upper bound for $\lambda_\e/c_\e^2$.
\begin{lem}\label{upbd}
It holds that
\begin{align*}
\limsup_{\e\to0}\frac{\lambda_\e}{c_\e^2}\leq\frac{\pi}{2(1-\beta)}e^{1+2\pi(1-\beta)A_0}.
\end{align*}
\end{lem}
\begin{proof}
First, we take a small $\delta>0$ such that $B^+_\delta(0)\subset\Om$. We write $B^+_r=B^+_r(0)$ for simplicity and define a function space
\begin{align*}
\mathcal{W}_{a,b}=\{u\in W^{1,2}(B^+_\delta\setminus B^+_{Rt_\e}):
u|_{\p B^+_\delta\setminus \p\mathbb{R}^2_+}=a, u|_{\p B^+_{Rt_\e}\setminus\p\mathbb{R}^2_+}=b,
\frac{\p u}{\p \nu}|_{\p B^+_\delta\setminus \p B_\delta}=0\}.
\end{align*}
It is well-known that $\inf_{u\in\mathcal{W}_\e(a_\e,b_\e)}\int_{B^+_\delta(0)\setminus B^+_{Rt_\e}(0)}|\D u|^2 dx$ is attained by the unique solution of
\begin{align*}
\begin{cases}
\Delta h=0\ \ \text{in}\ \ B^+_\delta\setminus \overline{B^+_{Rt_\e}},\\
h\in \mathcal{W}_{a_\e,b_\e}
\end{cases}
\end{align*}
and
\begin{align*}
h(x)=\frac{a_\e(\log|x|-\log(Rt_\e))+b_\e(\log\delta-\log|x|)}{\log\delta-\log(Rt_\e)}.
\end{align*}
Calculating directly,
\begin{align}\label{h-energy}
\int_{B^+_\delta(0)\setminus B^+_{Rt_\e}(0)} |\D h|^2 dx = \frac{\pi(a_\e-b_\e)^2}{\log\delta-\log(Rt_\e)}.
\end{align}
We denote
\begin{align*}
a_\e=\sup_{\p B^+_\delta\setminus \p\mathbb{R}^2_+}u_\e,\ \ \ \ b_\e=\inf_{\p B^+_{Rt_\e}\setminus\p\mathbb{R}^2_+}u_\e,
\ \ \ \widehat{u}_\e=\max\{a_\e, \min\{u_\e,b_\e\}\},
\end{align*}
then $\widehat{u}_\e\in \mathcal{W}_{a_\e,b_\e}$ and $|\D \widehat{u}_\e|\leq|\D u_\e|$
a.e. in $B_\delta^+\setminus B_{Rt_\e}^+$ for small $\e>0$. Hence
\begin{align}\label{energy-neck}
\int_{B^+_\delta\setminus B^+_{Rt_\e}} |\D h|^2 dx
\leq&\int_{B^+_\delta\setminus B^+_{Rt_\e}} |\D \widehat{u}_\e|^2 dx\\
\leq&\int_{B^+_\delta\setminus B^+_{Rt_\e}} |\D u_\e|^2 dx\nonumber\\
=& 1-\int_{\Om\setminus B^+_\delta}|\D u_\e|^2 dx-\int_{B^+_{Rt_\e}}|\D u_\e|^2 dx.\nonumber
\end{align}

Now we compute $\int_{\Om\setminus B_\delta^+}|\D u_\e|^2 dx$ and $\int_{B^+_{Rt_\e}} |\D u_\e|^2 dx$. By Lemma \ref{lem-green} and (\ref{green-exp}),
we have
\begin{align}\label{energy-out}
\int_{\Om\setminus B_\delta^+}|\D u_\e|^2 dx
=&\frac{1}{c_\e^2}\left(\int_{\Om\setminus B_\delta^+}|\D G|^2 dx+o_\e(1)\right)\\
=&\frac{1}{c_\e^2}\left(-\frac{1}{\pi}\log\delta+A_0+o_\delta(1)+o_\e(1)\right).\nonumber
\end{align}
Since $\varphi_\e\to\varphi_0$ in $W^{1,2}_{\text{loc}}(\mathbb{R}^2)$ as $\e\to0$, we have
\begin{align}\label{energy-in}
\int_{B_{Rt_\e}^+}|\D u_\e|^2 dx
=&\frac{1}{c_\e^2}\int_{B_R^+}|\D \varphi_\e(-t_\e x_\e+y)|^2dy\\
=&\frac{1}{c_\e^2}\left(\int_{B^+_R}|\D \varphi_0(y)|^2dy + o_\e(1)\right)\nonumber\\
=&\frac{1}{c_\e^2}\left(\frac{1}{\pi}\log R+\frac{1}{2\pi(1-\beta)}\log\frac{\pi}{2(1-\beta)}-\frac{1}{2\pi(1-\beta)}+o(1)\right),\nonumber
\end{align}
where $o(1)\to0$ as $\e\to0$ first and then $R\to+\infty$. It follows by Lemma \ref{lem-green} and (\ref{eq-varphi0}) that
\begin{align*}
a_\e=&\frac{1}{c_\e}\left(-\frac{1}{\pi}\log\delta+A_0+o_\delta(1)\right),\\
b_\e=&c_\e+\frac{1}{c_\e}\left(-\frac{1}{2\pi(1-\beta)}\log\left(1+\frac{\pi}{2(1-\beta)}R^{2(1-\beta)}\right)+o_\e(1)\right).
\end{align*}
Hence
\begin{align}
\pi(a_\e-b_\e)^2 = \pi c_\e^2-\frac{1}{1-\beta}\log\left(1+\frac{\pi}{2(1-\beta)}R^{2(1-\beta)}\right)+2\log\delta-2\pi A_0+o(1),
\end{align}\label{energy-neck-1}
where $o(1)\to0$ as $\e\to0$ first and then $\delta\to0$. Recalling the definition of $t_\e$, we have
\begin{align}\label{energy-neck-2}
\log\delta-\log(Rt_\e)
=\log\delta-\log R-\frac{1}{2(1-\beta)}\log\frac{\lambda_\e}{c_\e^2}+\frac{\pi(1-\beta-\e)c_\e^2}{1-\beta}.
\end{align}
Combining (\ref{h-energy})-(\ref{energy-neck-2}) one obtains
\begin{align*}
&\frac{\pi c_\e^2-\frac{1}{1-\beta}\log\left(1+\frac{\pi}{2(1-\beta)}R^{2(1-\beta)}\right)+2\log\delta-2\pi A_0+o(1)}
{\log\delta-\log R-\frac{1}{2(1-\beta)}\log\frac{\lambda_\e}{c_\e^2}+\frac{\pi(1-\beta-\e)}{1-\beta}c_\e^2}\\
\leq&1-\frac{1}{c_\e^2}\left(-\frac{1}{\pi}\log\delta+A_0+\frac{1}{\pi}\log R+\frac{1}{2\pi(1-\beta)}\log\frac{\pi}{2(1-\beta)}
-\frac{1}{2\pi(1-\beta)}+o(1)\right).
\end{align*}
It follows that
\begin{align*}
\frac{1+o(1)}{2(1-\beta)}\log\frac{\lambda_\e}{c_\e^2}\leq\frac{1}{2(1-\beta)}\log\frac{\pi}{2(1-\beta)}+\frac{1}{2(1-\beta)}+\pi A_0+o(1),
\end{align*}
which implies
\begin{align*}
\limsup_{\e\to0}\frac{\lambda_\e}{c_\e^2}\leq\frac{\pi}{2(1-\beta)}e^{1+2\pi(1-\beta)A_0}.
\end{align*}
This ends the proof of the lemma.
\end{proof}

\subsection{Completion of the proof (\ref{ineqsmt})}
We have already proved at the beginning of subsection 2.4 that: if $u_\e$ is bounded, then its weak limit $u_0$ attains the supremum in
(\ref{ineqsmt}) and we are done. On the other hand,
by Corollary \ref{cor} and Lemma \ref{upbd} one knows that: if $u_\e$ blows up, then
\begin{align}\label{upbound}
\sup_{u\in W^{1,2}(\Om),\int_\Om udx=0,||\D u||_2^2\leq1} \int_{\Om} \frac{e^{2\pi(1-\beta)u^2}}{|x|^{2\beta}}dx
=&\lim_{\e\to0}\int_{\Om} \frac{e^{2\pi(1-\beta-\e)u_\e^2}}{|x|^{2\beta}}dx\\
\leq& \int_{\Om}\frac{1}{|x|^{2\beta}}dx+\frac{\pi}{2(1-\beta)}e^{1+2\pi(1-\beta)A_0}.\nonumber
\end{align}
This completes the proof of (\ref{ineqsmt}).

\section{Extremal function}
In this section, we shall prove the existence of extremal function for
$$\sup_{u\in W^{1,2}(\Om), \int_\Om udx=0, \int_\Om |\D u|^2dx\leq1}\int_\Om\frac{e^{2\pi(1-\beta)u^2}}{|x|^{2\beta}}dx.$$
 If $c_\e$ is bounded, we are done. We assume $c_\e\to+\infty$
as $\e\to0$, then (\ref{upbound}) shows
\begin{equation}\label{functional}
\sup_{u\in W^{1,2}(\Om),\int_\Om udx=0,\int_\Om |\D u|^2dx\leq1} \int_{\Om} \frac{e^{2\pi(1-\beta)u^2}}{|x|^{2\beta}}dx
\leq\int_{\Om}\frac{1}{|x|^{2\beta}}dx+\frac{\pi}{2(1-\beta)}e^{1+2\pi(1-\beta)A_0}.
\end{equation}
To prove the existence of extremal function, we construct a sequence of functions $\phi_\e\in W^{1,2}(\Om)$
with $\int_\Om |\D \phi_\e|^2 dx=1$ such that
\begin{equation}\label{phi-low}
\int_\Om \frac{e^{2\pi(1-\beta)(\phi_\e-\bar{\phi_\e})^2}}{|x|^{2\beta}}dx > \int_{\Om}\frac{1}{|x|^{2\beta}}dx+\frac{\pi}{2(1-\beta)}e^{1+2\pi(1-\beta)A_0},
\end{equation}
provided $\e>0$ is sufficiently small. The contradiction between (\ref{functional}) and (\ref{phi-low}) tells us that $c_\e$ must be bounded.
Then we complete the proof of Theorem \ref{smt}.

Define $\phi_\e$ on $\Om$ by
\begin{align}\label{phi_ep}
\phi_\e(x)=
\begin{cases}
c+\frac{1}{c}\left(-\frac{1}{2\pi(1-\beta)}\log(1+\frac{\pi}{2(1-\beta)}(\frac{|x|}{\e})^{2(1-\beta)})+b\right),\ \ &x\in B_{R\e}^+(0)\\
\frac{1}{c}(G-\eta\psi),\ \ &x\in B_{2R\e}^+(0)\setminus B_{R\e}^+(0)\\
\frac{1}{c}G,\ \ &x\in \Om\setminus B_{2R\e}^+(0),
\end{cases}
\end{align}
where $G$ and $\psi$ are functions given as in (\ref{green-exp}), $R=(-\log\e)^{1/(1-\beta)}$, $\eta\in C_0^1(B_{2R\e}^+(0))$
is a cut-off function which satisfies that
$\eta\equiv1$ on $B_{R\e}^+(0)$ and $|\D\eta|\leq\frac{2}{R\e}$, $b$ and $c$ are constants depending only on $\e$ to be determined later.
To ensure $\phi_\e\in W^{1,2}(\Om)$, we set
\begin{align*}
c+\frac{1}{c}\left(-\frac{1}{2\pi(1-\beta)}\log(1+\frac{\pi}{2(1-\beta)}R^{2(1-\beta)})+b\right)
=\frac{1}{c}\left(-\frac{1}{\pi}\log(R\e)+A_0\right),
\end{align*}
which gives
\begin{align}\label{lipschitz}
c^2=-\frac{1}{\pi}\log\e+A_0-b+\frac{1}{2\pi(1-\beta)}\log\frac{\pi}{2(1-\beta)}+O(R^{-2(1-\beta)}).
\end{align}
By a direct calculation, we have
\begin{align}\label{phi-energy-in}
\int_{B^+_{R\e}(0)}|\D \phi_\e|^2dx
=\frac{1}{2\pi(1-\beta)c^2}\left(\log\frac{\pi}{2(1-\beta)}-1+\log R^{2(1-\beta)}+O(R^{-2(1-\beta)})\right).
\end{align}
Using (\ref{eq-green}) and (\ref{green-exp}) one has
\begin{align*}
&\int_{\Om\setminus B^+_{R\e}(0)}|\D G|^2 dx = -\frac{1}{\pi}\log(R\e)+A_0+O(R\e\log(R\e)),\\
&\int_{\Om\setminus B^+_{R\e}(0)}\D G\cdot \D(\eta\psi)dx = O(R\e),\\
&\int_{\Om\setminus B^+_{R\e}(0)}|\D(\eta\psi)|^2dx = O(R\e).
\end{align*}
Hence
\begin{align}\label{phi-energy-out}
  &\int_{\Om\setminus B^+_{R\e}(0)}|\D\phi_\e|^2dx\\
= &\frac{1}{c^2}\left(\int_{\Om\setminus B^+_{R\e}(0)}|\D G|^2 dx-2\int_{\Om\setminus B^+_{R\e}(0)}\D G\cdot \D(\eta\psi)dx+\int_{\Om\setminus B^+_{R\e}(0)}|\D(\eta\psi)|^2dx\right)\nonumber\\
= &\frac{1}{c^2}\left(-\frac{1}{\pi}\log(R\e)+A_0+O(R\e\log(R\e))\right).\nonumber
\end{align}
Combining (\ref{phi-energy-in}) and (\ref{phi-energy-out}) one obtains
\begin{align*}
\int_\Om |\D \phi_\e|^2dx
=\frac{1}{c^2}\left(-\frac{\log\e}{\pi}+\frac{\log\frac{\pi}{2(1-\beta)}}{2\pi(1-\beta)}-\frac{1}{2\pi(1-\beta)}+A_0+O(\frac{1}{R^{2(1-\beta)}})+O(R\e\log(R\e))\right).
\end{align*}
To ensure $\int_\Om |\D \phi_\e|^2dx=1$, we set
\begin{align}\label{deter-c}
c^2=-\frac{1}{\pi}\log\e+\frac{\log\frac{\pi}{2(1-\beta)}}{2\pi(1-\beta)}-\frac{1}{2\pi(1-\beta)}+A_0+O(R^{-2(1-\beta)})+O(R\e\log(R\e)).
\end{align}
Inserting (\ref{deter-c}) into (\ref{lipschitz}), we have
\begin{align}\label{deter-b}
b=\frac{1}{2\pi(1-\beta)}+O(R^{-2(1-\beta)})+O(R\e\log(R\e)).
\end{align}
By a direct calculation, one obtains
\begin{align}\label{phi-mean}
\bar{\phi_\e}=\frac{1}{|\Om|}\int_\Om \phi_\e dx
=\frac{1}{c}\left(O(R^2\e^2\log\e)+O(R^2\e^2\log R)+O(R^2\e^2\log(R\e))\right).
\end{align}
Now we estimate $\int_\Om \frac{e^{2\pi(1-\beta)(\phi_\e-\bar{\phi_\e})^2}}{|x|^{2\beta}}dx$.
In view of (\ref{deter-c}), (\ref{deter-b}) and (\ref{phi-mean}), one has on $B^+_{R\e}(0)$ that
\begin{align*}
2\pi(1-\beta)(\phi_\e-\bar{\phi_\e})^2
\geq&2\pi(1-\beta)c^2-2\log\left(1+\frac{\pi}{2(1-\beta)}(\frac{|x|}{\e})^{2(1-\beta)}\right)+4\pi(1-\beta)b\\
     &\ \ \ +O(R^2\e^2\log\e)+O(R^2\e^2\log R)+O(R^2\e^2\log(R\e))\\
=&-2(1-\beta)\log\e-2\log\left(1+\frac{\pi}{2(1-\beta)}(\frac{|x|}{\e})^{2(1-\beta)}\right)\\
  &\ \ \ +1+2\pi(1-\beta)A_0+\log\frac{\pi}{2(1-\beta)}\\
  &\ \ \ +O(R^{-2(1-\beta)})+O(R\e\log(R\e))+O(R^2\e^2\log\e)+O(R^2\e^2\log R).
\end{align*}
Therefore we obtain
\begin{align}\label{phi-funct-in}
\int_{B^+_{R\e}(0)} \frac{e^{2\pi(1-\beta)(\phi_\e-\bar{\phi_\e})^2}}{|x|^{2\beta}}dx
\geq&\frac{\pi}{2(1-\beta)}e^{1+2\pi(1-\beta)A_0}+O(R^{-2(1-\beta)})\\
    &+O(R\e\log(R\e))+O(R^2\e^2\log\e)+O(R^2\e^2\log R).\nonumber
\end{align}
Since
\begin{align*}
&\int_{B^+_{2R\e}(0)}|x|^{-2\beta}dx=O((R\e)^{2(1-\beta)})=O(\frac{1}{R^{2(1-\beta)}}),\\
&\int_{B^+_{2R\e}(0)}|x|^{-2\beta}(G-c\bar{\phi_\e})^2dx=O((R\e)^{2(1-\beta)}\log^2(R\e))=O(\frac{1}{R^{2(1-\beta)}}),
\end{align*}
we have
\begin{align}\label{phi-funct-out}
&\int_{\Om\setminus B^+_{2R\e}(0)}\frac{e^{2\pi(1-\beta)(\phi_\e-\bar{\phi_\e})^2}}{|x|^{2\beta}}dx\\
\geq&\int_{\Om\setminus B^+_{2R\e}(0)}\frac{1+2\pi(1-\beta)(\phi_\e-\bar{\phi_\e})^2}{|x|^{2\beta}}dx\nonumber\\
=&\int_{\Om\setminus B^+_{2R\e}(0)}|x|^{-2\beta}dx
  +\frac{2\pi(1-\beta)}{c^2}\int_{\Om\setminus B^+_{2R\e}(0)}|x|^{-2\beta}(G-c\bar{\phi_\e})^2dx\nonumber\\
=&\int_{\Om}|x|^{-2\beta}dx+\frac{2\pi(1-\beta)}{c^2}\int_{\Om}|x|^{-2\beta}(G-c\bar{\phi_\e})^2dx+O(\frac{1}{R^{2(1-\beta)}})\nonumber\\
=&\int_{\Om}|x|^{-2\beta}dx+\frac{2\pi(1-\beta)}{c^2}\int_{\Om}|x|^{-2\beta}G^2dx(1+o_\e(1))+O(\frac{1}{R^{2(1-\beta)}}).\nonumber
\end{align}
Recalling $R=(-\log\e)^{1/(1-\beta)}$, one has $R^{-2(1-\beta)}=o(\frac{1}{c^2})$, then (\ref{phi-energy-in}) and (\ref{phi-energy-out}) tell us that
\begin{align*}
    &\int_{\Om} \frac{e^{2\pi(1-\beta)(\phi_\e-\bar{\phi_\e})^2}}{|x|^{2\beta}}dx\\
\geq&\int_{B^+_{R\e}(0)} \frac{e^{2\pi(1-\beta)(\phi_\e-\bar{\phi_\e})^2}}{|x|^{2\beta}}dx
    +\int_{\Om\setminus B^+_{2R\e}(0)}\frac{e^{2\pi(1-\beta)(\phi_\e-\bar{\phi_\e})^2}}{|x|^{2\beta}}dx\\
=&\int_{\Om}\frac{1}{|x|^{2\beta}}dx+\frac{\pi}{2(1-\beta)}e^{1+2\pi(1-\beta)A_0}+\frac{2\pi(1-\beta)}{c^2}\int_{\Om}|x|^{-2\beta}G^2dx+o(\frac{1}{c^2})\\
>&\int_{\Om}\frac{1}{|x|^{2\beta}}dx+\frac{\pi}{2(1-\beta)}e^{1+2\pi(1-\beta)A_0}
\end{align*}
when $\e>0$ is sufficiently small.


\vskip 10pt


\begin{thebibliography}{XXX}

\bibitem[AS]{AS}
Adimurthi, Sandeep, K., A singular Moser-Trudinger embedding and its applications, NoDEA Nonlinear Differential Equations Appl. 13 (2007), no. 5-6, 585-603.

\bibitem[B]{B}
Berger,

\bibitem[CC]{CC}
Carleson, L., Chang, Sun-Yung A., On the existence of an extremal function for an inequality of J. Moser, Bull. Sci. Math. (2) 110 (1986), no. 2, 113-127.

\bibitem[Ch]{Ch}
Chen, W., A Trudinger inequality on surfaces with conical singularities, Proc. Amer. Math. Soc. 108 (1990), no. 3, 821-832.

\bibitem[CL]{CL}
Chen, W., Li, C., Classification of solutions of some nonlinear elliptic equations, Duke Math. J. 63 (1991), 615-622.

\bibitem[CL2]{CL2}
Chen, W., Li, C., What kinds of singular surfaces can admit constant curvature? Duke Math. J. 78 (1995), 437-451.


\bibitem[CL3]{CL3}
Chen, W., Li, C., Prescribing Gaussian curvatures on surfaces with conical singularities, J. Geom. Anal. 1 (1991), 359-372.

\bibitem[CY]{CY87}
Chang, Sun-Yung A., Yang, Paul C.,  Prescribing Gaussian curvatures on $\mathbb{S}^2$, Acta Math. 159 (1987), 214-259.

\bibitem[CY2]{CY88}
Chang, Sun-Yung A., Yang, Paul C., Conformal deformation of metrics on $\mathbb S^2$, J. Differential Geom. 27 (1988), no. 2, 259-296.

\bibitem[ChL]{ChL}
Chang, K., Liu, J., On Nirenberg's problem, Internat. J. Math. 4 (1993), 35-58.

\bibitem[CD]{CD}
Chen, W., Ding, W., Scalar curvatures on $\mathbb S^2$, Trans. Amer. Math. Soc. 303 (1987), 365-382.



\bibitem[CR]{CR}
Csat\'{o}, G., Roy, P., Extremal functions for the singular Moser-Trudinger inequality in 2 dimensions, Calc. Var. Partial Differential Equations 54 (2015), no. 2, 2341-2366.

\bibitem[DJLW]{DJLW}
Ding, W., Jost, J., Li, J., Wang, G., The differential equation $\Delta u =8\pi-8\pi he^u$ on a compact Riemann surface, Asian J. Math. 1 (1997), no. 2, 230-248.

\bibitem[F]{F}
Flucher, M., Extremal functions for the Trudinger-Moser inequality in $2$ dimensions, Comment. Math. Helv. 67 (1992), no. 3, 471-497.

\bibitem[Fo]{Fo}
Fontana, L., Sharp borderline Sobolev inequalities on compact Riemannian manifolds, Comment. Math. Helv. 68 (1993), no. 3, 415-454.

\bibitem[KW]{KW}
Kazdan, J., Warner, F., Curvature functions for compact 2-manifolds, Ann. Math. 99 (1974), 14-47.

\bibitem[Li]{Li}
Li, Y., Moser-Trudinger inequality on compact Riemannian manifolds of dimension two, J. Partial Differential Equations 14 (2001), no. 2, 163-192.

\bibitem[Li2]{Li2}
Li, Y., Extremal functions for the Moser-Trudinger inequalities on compact Riemannian manifolds, Sci. China Ser. A 48 (2005), no. 5, 618-648.

\bibitem[LL]{LL}
Li, Y., Liu, P., A Moser-Trudinger inequality on the boundary of a compact Riemann surface, Math. Z. 250 (2005), no. 2, 363-386.

\bibitem[LLY]{LLY}
Li, Y., Liu, P., Yang, Y., Moser-Trudinger inequalities of vector bundle over a compact Riemannian manifold of dimension 2, Calc. Var. Partial Differential Equations 28 (2007), no. 1, 59-83.

\bibitem[M]{M}
Moser, J., A sharp form of an inequality by N. Trudinger, Indiana Univ. Math. J. 20 (1970/71), 1077-1092.

\bibitem[M2]{M2}
Moser, J., On a nonlinear problem in differential geometry, Dynamical systems, Academic Press, New York, 1973.

\bibitem[PT]{PT}
Prajapat, J., Tarantello, G., On a class of elliptic problems in $\mathbb{R}^2$: symmetry and uniqueness results, Proc. Roy. Soc. Edinburgh Sect. A 131 (2001), no. 4, 967-985.

\bibitem[St]{St}
Struwe, M., Critical points of embeddings of $H^{1,n}_0$ into Orlicz spaces, Ann. Inst. H. Poincar\'{e} Anal. Non Lin\'{e}aire 5 (1988), no. 5, 425-464.

\bibitem[St2]{St2} Struwe, M., A flow approach to Nirenberg's problem, Duke Math. J. 128 (2005), no. 1, 19-64.

\bibitem[Tro]{Tro}
Troyanov, M., Prescribing curvature on compact surfaces with conical singularities, Trans. Amer. Math. Soc. 324 (1991), no. 2, 793-821.

\bibitem[Tr]{Tr}
Trudinger, Neil S., On imbeddings into Orlicz spaces and some applications, J. Math. Mech. 17 (1967) 473-483.

\bibitem[Y]{Y06}
Yang, Y., Extremal functions for Moser-Trudinger inequalities on 2-dimensional compact Riemannian manifolds with boundary, Internat. J. Math. 17 (2006), no. 3, 313-330.

\bibitem[Y2]{Y07}
Yang, Y.,
Moser-Trudinger inequality for functions with mean value zero,
Nonlinear Analysis 66(2007), 2742-2755.

\bibitem[Y3]{Y2007}
 Yang, Y., A sharp form of the Moser-Trudinger inequality on a compact Riemannian surface, Trans. Amer. Math. Soc. 359 (2007), no. 12, 5761-5776.

\bibitem[Y4]{Y15}
Yang, Y., Extremal functions for Trudinger-Moser inequalities of Adimurthi-Druet type in dimension two, J. Differential Equations 258 (2015), no. 9, 3161-3193.

\bibitem[YZ]{YZ}
Yang, Y., Zhu, X., Blow-up analysis concerning singular Trudinger-Moser inequalities in dimension two, J. Funct. Anal. 272 (2017), no. 8, 3347-3374.

\bibitem[YZ2]{YZ2}
Yang, Y., Zhu, X., Prescribing Gaussian curvature on closed Riemann surface with conical singularity in the negative case, Ann. Acad. Sci. Fenn. Math. 44 (2019), no. 1, 167-181.

\bibitem[Z]{Z}
Zhu, X., A generalized Trudinger-Moser inequality on a compact Riemannian surface with conical singularities, Sci. China Math. 62 (2019), no. 4, 699-718.

\end{thebibliography}
\end{document}